\documentclass[12pt,oneside]{article}
\usepackage{amsmath,amssymb,amsfonts,amsthm}
\usepackage{color}

\newcommand{\A}{{\cal A}}

\newcommand{\T}{{\cal T}}

\newcommand{\Real}{\mathbb R}

\newcommand{\To}{\longrightarrow}

\newcommand{\p}{\pi^{-1}(TM)}
\newcommand {\cp}{\mathfrak{X}(\pi (M))}

\def\o#1{\overline{#1}}

\def\Section#1{\vspace{30truept}\addtocounter{section}{1}\setcounter{thm}{0}\setcounter{equation}{0}
{\noindent\Large\bf\arabic{section}.~~#1}\par \vspace{12pt}}

\newtheorem{thm}{Theorem}[section]

\newtheorem{lem}[thm]{Lemma}
\newtheorem{prop}[thm]{Proposition}
\newtheorem{defn}[thm]{Definition}

\newtheorem{rem}[thm]{Remark}

\numberwithin{equation}{section}

\begin{document}
\title{{\bf  CONFORMAL CHANGE OF SPECIAL FINSLER SPACES}}
\author{{\bf Nabil L. Youssef\,$^{\dagger}$,
 S. H. Abed$^{\ddagger}$ and A. Soleiman$^{\sharp}$}}
\date{}

\maketitle                     
\vspace{-1.15cm}
\begin{center}
{$^{\dag}$Department of Mathematics, Faculty of Science,\\ Cairo
University, Giza, Egypt}
\end{center}
\vspace{-0.8cm}
\begin{center}
nlyoussef2003@yahoo.fr,\ nyoussef@frcu.eun.eg
\end{center}

\begin{center}
{$^{\ddag}$Department of Mathematics, Faculty of Science,\\ Cairo
University, Giza, Egypt}
\end{center}
\vspace{-0.8cm}
\begin{center}
sabed52@yahoo.fr,\ sabed@frcu.eun.eg
\end{center}

\begin{center}
{$^{\sharp}$Department of Mathematics, Faculty of Science,\\ Benha
University, Benha,
 Egypt}
\end{center}
\vspace{-0.8cm}
\begin{center}
soleiman@mailer.eun.eg
\end{center}

\vspace{1cm} \maketitle
\smallskip

\noindent{\bf Abstract.} The present paper is a continuation of
 a foregoing paper [Tensor, N. S., 69 (2008), 155-178]. The main
 aim  is to establish \emph{an
intrinsic investigation} of the  conformal change of the most
important special Finsler spaces, namely, $C^{h}$-recurrent,
$C^{v}$-recurrent, $C^{0}$-recurrent, $C_{2}$-like,
quasi-$C$-reducible, $C$-reducible, Berwald space,
$S^{v}$-recurrent, $P^*$-Finsler manifold, $R_{3}$-like,
$P$-symmetric, Finsler manifold of $p$-scalar curvature and Finsler
manifold of $s$-$ps$-curvature. Necessary and
 sufficient  conditions for such special Finsler manifolds to be
invariant under a conformal change are obtained.   Moreover, the
conformal change of Chern and Hashiguchi connections, as well as
their curvature tensors, are given.

\bigskip
\medskip\noindent{\bf Keywords.\/}\, Conformal change,  $C^{h}$-recurrent, $C^{v}$-recurrent,
$C^{0}$-recurrent, $C_{2}$-like, quasi-$C$-reducible, $C$-reducible,
Berwald space, $S^{v}$-recurrent, $P^*$-Finsler manifold,
$R_{3}$-like, $P$-symmetric, Chern connection,  Hashiguchi
connection.

\bigskip
\medskip\noindent{\bf  2000 Mathematics Subject Classification.\/} 53C60,
53B40.

\newpage

\vspace{30truept}\centerline{\Large\bf{Introduction}}\vspace{12pt}
\par
 Studying Finsler geometry one encounters substantial
difficulties trying to seek analogues of classical global, or
sometimes even local, results of Riemannian geometry. These
difficulties arise mainly from the fact that in Finsler geometry all
geometric objects  depend not only on positional coordinates, as in
Riemannian geometry, but also on directional arguments.
\par
The infinitesimal transformations in Riemannian and  Finsler
geometries  are important, not only in differential geometry,  but
also in application to other branches of science, especially in
the process of geometrization of physical theories.
\par
The theory of conformal changes in Riemannian geometry has been
deeply studied (\emph{locally and intrinsically}) by many authors.
As regards to Finsler geometry, an almost complete \emph{local} theory of
conformal changes has been established (\cite{r60}, \cite{r23}, \cite{Balkan.2},
\cite{r5}, \cite{r6}, \cite{Balkan.1}, \cite{r73},
, $\cdots$, etc.).
\par
In \cite{r62}, we investigated intrinsically conformal changes in
Finsler geometry, where we got, among other results, a
characterization of conformal changes. Also the conformal change
of Barthel connection and its curvature tensor were studied.
Moreover, the conformal changes of Cartan and Berwald connections
as well as their curvature tensors, were obtained.
\par
The present paper is a continuation of ~\cite{r62} where we
present an \emph{intrinsic} theory of conformal changes of
special Finsler spaces. Moreover, we study the conformal change of
Chern and Hashiguchi connections
\par
The  paper consists of two parts preceded by an introductory
section $(\S 1)$, which provides a brief account of the basic
definitions and concepts necessary for this work.
\par
In the first part $(\S 2)$, the conformal  change of Chern
 and  Hashiguchi connections, as well as their curvature
tensors, are given.
\par
In the second part $(\S 3)$,  we provide an intrinsic
investigation  of the conformal change of the most important
special Finsler spaces, namely, $C^{h}$-recurrent,
$C^{v}$-recurrent, $C^{0}$-recurrent, $S^{v}$-recurrent,
$P^*$-manifold, $R_{3}$-like, Finsler manifold
 of $p$-scalar curvature and of $s$-$ps$-curvature, $\cdots,$ etc.
  Moreover, we obtain  necessary and sufficient
 conditions for such special Finsler manifolds to be
invariant under a conformal change.
\par
Finally, it should be noted that the present work is formulated in
a prospective modern coordinate-free form, without being trapped
into the complications of indices. However, some important results
of \cite{r23}, \cite{r5}, \cite{r6} and others (obtained in local
coordinates) are immediately derived from the obtained global
results (when localized).

\Section{Notation and Preliminaries}

In this section, we give a brief account of the basic concepts
 of the pullback approach to global Finsler geometry necessary for this work. For more
 detail, we refer to \cite{r58}, \cite{r91} and \cite{r48}.
 We assume that all geometric objects treated are of class
$C^{\infty}$. The
following notations are to be used throughout this paper:\\
 $M$: a real paracompact differentiable manifold of finite dimension $n$ and of
class $C^{\infty}$,\\
 $\mathfrak{F}(M)$: the $\Real$-algebra of differentiable functions
on $M$,\\
 $\mathfrak{X}(M)$: the $\mathfrak{F}(M)$-module of vector fields
on $M$,\\
$\pi_{M}:TM\longrightarrow M$: the tangent bundle of $M$,\\
$\pi^{*}_{M}:T^{*}M\longrightarrow M$: the cotangent bundle of $M$,\\
$\pi: \T M\longrightarrow M$: the subbundle of nonzero vectors
tangent to $M$,\\
$V(TM)$: the vertical subbundle of the bundle $TTM$,\\
 $P:\pi^{-1}(TM)\longrightarrow \T M$: the pullback of the
tangent bundle $TM$ by $\pi$,\\
$P^{*}:\pi^{-1}(T^{*}M)\longrightarrow \T M$: the pullback of the
cotangent bundle $T^{*}M$ by $\pi$,\\
 $\mathfrak{X}(\pi (M))$: the $\mathfrak{F}(TM)$-module of
differentiable sections of  $\pi^{-1}(T M)$,\\
 $\mathfrak{X}^{*}(\pi (M))$: the $\mathfrak{F}(TM)$-module of
differentiable sections of  $\pi^{-1}(T^{*} M)$,\\
$ i_{X}$: the interior product with respect to  $X
\in\mathfrak{X}(M)$,\\
$df$: the exterior derivative  of $f\in \mathfrak{F}(M)$,\\
$ d_{L}:=[i_{L},d]$, $i_{L}$ being the interior derivative with
respect to a vector form $L$.

\par Elements  of  $\mathfrak{X}(\pi (M))$ will be called
$\pi$-vector fields and will be denoted by barred letters
$\overline{X} $. Tensor fields on $\pi^{-1}(TM)$ will be called
$\pi$-tensor fields. The fundamental $\pi$-vector field is the
$\pi$-vector field $\overline{\eta}$ defined by
$\overline{\eta}(u)=(u,u)$ for all $u\in \T M$.

We have the following short exact sequence of vector
bundles\vspace{-0.1cm}
$$0\longrightarrow
 \pi^{-1}(TM)\stackrel{\gamma}\longrightarrow T(\T M)\stackrel{\rho}\longrightarrow
\pi^{-1}(TM)\longrightarrow 0 ,\vspace{-0.1cm}$$
 where the bundle morphisms $\rho$ and $\gamma$ are defined respectively by
$\rho := (\pi_{\T M},d\pi)$ and $\gamma (u,v):=j_{u}(v)$, $j_{u}$
being the natural isomorphism $j_{u}:T_{\pi_{M}(u)}M \longrightarrow
T_{u}(T_{\pi_{M}(u)}M)$. The vector $1$-form $J$ on $TM$ defined by
$J:=\gamma\circ\rho$ is called the natural almost tangent structure
of $T M$. The vertical vector field $\mathcal{C}$ on $TM$ defined by
$\mathcal{C}:=\gamma\circ\overline{\eta} $ is called  the canonical
or Liouville vector field.
\par
Let $D$ be  a linear connection (or simply a connection) on the
pullback bundle $\pi^{-1}(TM)$. The map \vspace{-0.1cm}
$$K:T \T M\longrightarrow \pi^{-1}(TM):X\longmapsto D_X \overline{\eta}
\vspace{-0.1cm}$$ is  called the connection map or the deflection
map associated with $D$. A tangent vector $X\in T_u (\T M)$ is
said to be horizontal if $K(X)=0$. The vector space $H_u (\T M)$
of the horizontal vectors
 at $u \in  \T M$ is called the horizontal space of $M$ at $u$.
   The connection $D$ is said to be regular if
\begin{equation}\label{direct sum}
T_u (\T M)=V_u (\T M)\oplus H_u (\T M) \qquad \forall u\in \T M .
\end{equation}
\par If $M$ is endowed with a regular connection $D$, then the  maps
\begin{eqnarray*}
   \gamma &:& \pi^{-1}(T M)  \To V(\T M), \\
    \rho |_{H(\T M)}&:&H(\T M)\To \pi^{-1}(TM),\\
     K |_{V(\T M)}&:&V(\T M) \To \pi^{-1}(T M)
\end{eqnarray*}
are vector bundle isomorphisms.
   Let $\beta:=(\rho |_{H(\T M)})^{-1}$, called the horizontal map associated with $D$, then \vspace{-0.2cm}
   \begin{align}\label{fh1}
    \rho\circ\beta = id_{\pi^{-1} (TM)}, \quad  \quad
       \beta\circ\rho =\left\{
                                \begin{array}{ll}
                                          id_{H(\T M)} & {\,\, on\,\,   H(\T M)} \\
                                         0 & {\,\, on \,\,   V(\T M)}
                                       \end{array}
                                     \right.\vspace{-0.2cm}
\end{align}
\par
 The (classical)  torsion tensor $\textbf{T}$  of the connection
$D$ is given by
$$\textbf{T}(X,Y)=D_X \rho Y-D_Y\rho X -\rho [X,Y] \quad
\forall\,X,Y\in \mathfrak{X} (\T M),$$ from which the horizontal or
(h)h-torsion tensor $Q$ and the mixed or (h)hv-torsion tensor $T$
are defined respectively by \vspace{-0.2cm}
$$Q (\overline{X},\overline{Y})=\textbf{T}(\beta \overline{X}\beta \overline{Y}),
\, \,\, T(\overline{X},\overline{Y})=\textbf{T}(\gamma
\overline{X},\beta \overline{Y}) \quad \forall \,
\overline{X},\overline{Y}\in\mathfrak{X} (\pi (M)).\vspace{-0.2cm}$$
\par
The (classical) curvature tensor  $\textbf{K}$ of the connection $D$
is given by
 $$ \textbf{K}(X,Y)\rho Z=-D_X D_Y \rho Z+D_Y D_X \rho Z+D_{[X,Y]}\rho Z
  \quad \forall\, X,Y, Z \in \mathfrak{X} (\T M),$$
from which the horizontal (h-), mixed (hv-) and vertical (v-)
curvature tensors, denoted by $R$, $P$ and $S$ respectively, are
defined by
$$R(\overline{X},\overline{Y})\overline{Z}=\textbf{K}(\beta
\overline{X}\beta \overline{Y})\overline{Z},\quad
P(\overline{X},\overline{Y})\overline{Z}=\textbf{K}(\beta
\overline{X},\gamma \overline{Y})\overline{Z},\quad
S(\overline{X},\overline{Y})\overline{Z}=\textbf{K}(\gamma
\overline{X},\gamma \overline{Y})\overline{Z}.$$ The contracted
curvature tensors  $\widehat{R}$, $\widehat{P}$ and $\widehat{S}$,
also known as the
 (v)h-, (v)hv- and (v)v-torsion tensors, are defined by
$$\widehat{R}(\overline{X},\overline{Y})={R}(\overline{X},\overline{Y})\overline{\eta},\quad
\widehat{P}(\overline{X},\overline{Y})={P}(\overline{X},\overline{Y})\overline{\eta},\quad
\widehat{S}(\overline{X},\overline{Y})={S}(\overline{X},\overline{Y})\overline{\eta}.$$
\par If $M$ is endowed with a metric $g$ on $\p$, we write
\begin{equation}\label{cur.g}
    R(\overline{X},\overline{Y},\overline{Z}, \overline{W}):
=g(R(\overline{X},\overline{Y})\overline{Z}, \overline{W}),\,
\cdots, \, S(\overline{X},\overline{Y},\overline{Z}, \overline{W}):
=g(S(\overline{X},\overline{Y})\overline{Z}, \overline{W}).
\end{equation}

On a Finsler manifold $(M,L)$, there are {\em{{canonically}}}
associated  four linear connections on $\pi^{-1}(TM)$ \cite{r94}:
the Cartan connection $\nabla$, the Chern (Rund) connection $D^{c}$,
the Hashiguchi connection $ D^{*} $ and   the Berwald connection
$D^{\circ} $. Each  of these  connections is  regular with
(h)hv-torsion $T$ satisfying $T(\o X, \o \eta)=0$.

\begin{defn}  Let $(M,L)$ be a Finsler
manifold and  $g$ the Finsler metric defined by $L$. Let $T$ be the
(h)hv-torsion tensor and  $S$, $P$, $R$ are  the $v$-, $hv$- and
$h$-curvature tensors associated with the Cartan connection
$\nabla$. We define\vspace{-0.2cm}
\begin{eqnarray*}
\ell(\overline{X})&:=&L^{-1}g(\overline{X},\overline{\eta}),\\
\hbar&:=& g-\ell \otimes \ell:\text{the  angular metric tensor},\\
T(\overline{X},\overline{Y},\overline{Z})&:
=&g(T(\overline{X},\overline{Y}),\overline{Z}):\text{the  Cartan
 tensor},\\
    C(\overline{X})&:=& Tr\{\overline{Y} \longmapsto
T(\overline{X},\overline{Y})\}:\text{the  contracted torsion},\\
 g(\overline{C}, \overline{X})&:=&C(\overline{X}):\text{$\overline{C}$
 is the $\pi$-vector field associated with the $\pi$-form $C$, by duality}, \\
 Ric^v(\overline{X},\overline{Y})&:=& Tr\{
\overline{Z} \longmapsto S(\overline{X},\overline{Z})\overline{Y}\}:\text{the  vertical Ricci tensor},\\
Ric^h(\overline{X},\overline{Y})&:=& Tr\{ \overline{Z} \longmapsto
R(\overline{X},\overline{Z})\overline{Y}\}:\text{the horizontal Ricci tensor},\\
g(Ric_0^v(\overline{X}),\overline{Y})&:=&Ric^v(\overline{X},\overline{Y}):\text{the vertical Ricci  map} \  Ric_0^v,\\
g(Ric_0^h(\overline{X}),\overline{Y})&:=&Ric^h(\overline{X},\overline{Y}):\text{the horizontal Ricci  map} \ Ric_0^h,\\
 Sc^v&:=&\text {Tr}\{ \o X \longmapsto Ric_0^v\o X\}:\text{the vertical scalar curvature}, \\
 Sc^h&:=& \text {Tr}\{ \o X \longmapsto Ric_0^h(\o X)\}:\text{the horizontal scalar curvatures}.
\end{eqnarray*}
\end{defn}

The following two  results \cite{r94} give an explicit expression
for each of  the  Berwald, Chern and Hashiguchi connections  in
terms of the Cartan connection $\nabla$. \vspace{-0.2cm}
\begin{thm} \label{th.r5}The Chern connection $D^{c}$ is
 given, in terms of  Cartan connection,~by
   $$D^{c}_{X}\overline{Y} = \nabla _{X}\overline{Y}
- T(KX,\overline{Y})= D^{\circ} _{X}\overline{Y} -{\widehat{P}}(\rho
X, \overline{Y}). \vspace{-0.2cm}$$ In particular, we have
\begin{description}
  \item[(a)] $D^{c}_{\gamma \overline{X}}\overline{Y}=\nabla _{\gamma
  \overline{X}}\overline{Y}-T(\overline{X},\overline{Y})=D^{\circ} _{\gamma
  \overline{X}}\overline{Y}$.

 \item[(b)] $D^{c}_{\beta \overline{X}}\overline{Y}=\nabla _{\beta
  \overline{X}}\overline{Y}=D^{\circ} _{\beta
  \overline{X}}\overline{Y}-{\widehat{P}}(\overline{X}, \overline{Y}).$
\end{description}
\end{thm}

\begin{thm}\label{th.t12} The Hashiguchi  connection $ D^{*} $  is
given, in terms of Cartan connection,~by \vspace{-0.2cm}
  \begin{eqnarray*}
   D^{*}_{X}\overline{Y} = \nabla _{X}\overline{Y}
+ {\widehat{P}}(\rho X,\overline{Y}) =D^{\circ}_{X}\overline{Y}
+{T}(K X, \overline{Y}). \vspace{-0.2cm}
  \end{eqnarray*}
In particular, we have
\begin{description}
  \item[(a)] $ D^{*}_{\gamma \overline{X}}\overline{Y}=\nabla _{\gamma
  \overline{X}}\overline{Y}=D^{\circ}_{\gamma \overline{X}}\overline{Y}
+{T}( \overline{X}, \overline{Y})$.

 \item[(b)] $ D^{*}_{\beta \overline{X}}\overline{Y}=\nabla _{\beta
  \overline{X}}\overline{Y}+\widehat{P}(\overline{X},\overline{Y})
  =D^{\circ}_{\beta \overline{X}}\overline{Y}.$
\end{description}
\end{thm}

\par Now, we give some concepts and results concerning the Klein-Grifone approach
to intrinsic Finsler geometry. For more details, we refer to
\cite{r21}, \cite{r22} and  \cite{r27}.

\begin{prop}\label{ch.6pp} Let $(M,L)$ be a Finsler manifold. The vector field
$G$ on $TM$ defined by $i_{G}\,\Omega =-dE$ is a spray, where
 $E:=\frac{1}{2}L^{2}$ is the energy function and $\Omega:=dd_{J}E$.
 Such a spray is called the canonical spray.
 \end{prop}

A nonlinear connection on $M$ is a vector $1$-form $\Gamma$ on $TM$,
$C^{\infty}$ on $\T M$, such that $J \Gamma=J, \,\, \Gamma J=-J .$
The horizontal and vertical projectors
 associated with $\Gamma$ are
defined by
   $h:=\frac{1}{2} (I+\Gamma)$ and $v:=\frac{1}{2}
 (I-\Gamma)$ respectively.   The torsion and  curvature of
$\Gamma$ are defined by $t:=\frac{1}{2}[J,\Gamma]$ and
 $\mathfrak{R}:=-\frac{1}{2}[h,h]$  respectively. A nonlinear
 connection $\Gamma$ is homogenous if $[\mathcal{C}, \Gamma]=0$. It  is conservative
 if $d_{h}E=0$.

\begin{thm}  \label{c.ba.} On a Finsler manifold $(M,L)$, there exists a unique
conservative homogenous nonlinear  connection  with zero torsion. It
is given by\,{\em:} \vspace{-0.2cm} $$\Gamma =
[J,G],\vspace{-0.2cm}$$ where $G$ is the canonical spray.
 Such a nonlinear connection is called the canonical connection,
 the Cartan nonlinear connection or the Barthel connection associated with $(M,L)$.
\end{thm}

It can be proved \cite{r94} that the nonlinear connection
associated with each of the four canonical linear connections
coincide with the Barthel connection.

\bigskip

 We terminate this section by the following fact. Under an  arbitrary change $L \To \widetilde{L}$ of Finsler
structures on $M$, let  the corresponding Cartan connections
$\nabla$ and $\widetilde{\nabla} $ be related by
$$\widetilde{\nabla} _{X} \overline{Y} =\nabla _{X}
\overline{Y}+\omega (X,\overline{Y}).$$ If we denote

\begin{equation}\label{ABN}
\begin{array}{rcl}
A(\overline{X},\overline{Y})&:=&\omega(\gamma
\overline{X},\overline{Y}),\ \
B(\overline{X},\overline{Y}):=\omega(\beta
\overline{X},\overline{Y}),\ \\
N(\overline{X})&:=&B(\overline{X},\overline{\eta}),\qquad\qquad
N_{o}:= N (\overline{\eta}),
\end{array}
\end{equation}
then we have

\begin{prop}\emph{\cite{r62}}\label{pp.1} Under an  arbitrary change $L \To \widetilde{L}$
of Finsler structures on $M$,
 the corresponding Barthel connections $\Gamma$ and $\widetilde{\Gamma}
  $  are related by \vspace{-0.1cm}
  \begin{equation*}
    \widetilde{\Gamma} = \Gamma -2 \L,\, \text{with}\ \ \L:= \gamma o N o \rho. \vspace{-0.1cm}
\end{equation*}
  Moreover, we have\ \ $ \widetilde{h} =h - \L ,\ \  \widetilde{v} =v + \L .$
\end{prop}

\Section{Conformal change of the fundamental regular  connections
and their curvature tensors}

In this section, we first review some concepts and results
concerning the conformal changes of the Cartan and Berwald
connections \cite{r62}. Then, using these results, the conformal
changes of Chern and Hashiguchi connections,  as well as their
curvature tensors, are investigated.

\begin{defn}
Let $(M,L)$ and $(M,\widetilde {L})$ be two Finsler manifolds. The
two associated metrics $g$ and $\widetilde {g}$ are said to be
conformal if there exists a positive differentiable function
$\sigma(x)$  such that $\widetilde
{g}(\overline{X},\overline{Y})=e^{2\sigma(x)}
g(\overline{X},\overline{Y})$.  Equivalently, $g$ and
$\widetilde{g}$ are conformal iff $\widetilde
{L}^2=e^{2\sigma(x)}L^2 .$ In this case, the transformation $L\To
\widetilde {L}$ is said to be a conformal transformation and the
two Finsler manifold $(M,L)$ and $(M,\widetilde{L})$ are said to
be conformal or conformally related.
\end{defn}

\begin{defn}Let $(M,L)$ and $(M,\widetilde{L})$ be two conformal Finsler
manifolds  with $\widetilde{g}=e^{2\sigma(x)}g$. \\
\emph{\textbf{(a)}} A geometric object $W$ is said to be conformally
invariant
    {\em(resp. conformally $ \sigma$-invariant)}
    if  $\widetilde{W}=W$ {\em(resp.
    $\widetilde{W}=e^{2\sigma(x)}W$)}.\\
\emph{\textbf{(b)}}  A property $\xi$ is said to be a conformal
invariant property if whenever it is possessed
    by $(M,L)$, it is also possessed by $(M,\widetilde{L})$.
 \end{defn}

\begin{defn}\label{grad} The vertical gradient of a function $f\in \mathfrak{F}(TM)$, denoted
$grad_{v}f$,  is the vertical vector field
 $JX$ defined by \vspace{-0.1cm}
 $$df(Y)=\bar{g}(JX,JY),\ \ \text{for all }\ Y{\!}\in
\mathfrak{X}(TM),\vspace{-0.1cm}$$ where $\bar{g}$ is the metric
on $V(TM)$ defined in~{\em\cite{r21}}.\vspace{-0.1cm}
$$\bar{g}(JY,JZ)=\Omega(JY,Z),\ \ \text{for all }\ Y, Z\in \mathfrak{X}(TM).$$
\end{defn}

\begin{lem}\label{lem.c1} Let $(M,L)$ and $(M,\widetilde {L})$ be  conformally related  Finsler manifolds with
$\widetilde{g}=e^{2 \sigma(x)}g $.  The associated
  Barthel connections $\Gamma$ and  $\widetilde{\Gamma}$ are related by
\vspace{-0.1cm}
\begin{equation}\label{eq.1a}
    \begin{array}{rcl}
        \widetilde{\Gamma}& = &\Gamma - 2\L , \\
    \text{where} \qquad\qquad \L &:=& d\sigma \otimes \mathcal{C}+ \sigma_{1}J - d_{J}E \otimes
grad_{v}\sigma - EF= \gamma o N o \rho,
 \vspace{-0.2cm}
    \end{array}
\end{equation}
\noindent  $\sigma_{1}:=d_{G}\sigma $ and  $F:=[J,grad_{v}\sigma]$.\\
      Consequently,
$\widetilde{h} = h - \L $, \,\,  $\widetilde{v} = v + \L $ or,
equivalently, \,$\widetilde{\beta} = \beta - \L o \beta  $, \,\,
$\widetilde{K} = K + K o \L $.
 \end{lem}

Concerning the conformal change of the Cartan and Berwald
connections and their curvature tensors, we have the following two
results \cite{r62}.
 \vspace{-0.1cm}
\begin{thm}\label{th.c1} If $(M,L)$ and $(M,\widetilde {L})$ are conformally related Finsler
manifolds, then the associated Cartan  connections $\nabla$ and
  $\widetilde{\nabla}$  are related by:\vspace{-0.2cm}
\begin{equation*}\label{eq.1}
\widetilde {\nabla} _{X}\overline{Y} =
    \nabla _{X}\overline{Y}+ \omega(X, \overline{Y}), \vspace{-0.1cm}
 \end{equation*}
where \vspace{-0.1cm}
 \begin{equation*}\label{eq.t6}
    \begin{split}
     \omega(X,\overline{Y}):=&(h X \cdot \sigma (x)) \overline{Y}
     + (\beta \overline{Y} \cdot \sigma (x)) \rho X -g(\rho X,
     \overline{Y}) \overline{P} \\
     &
     - T(N \overline{Y},\rho X) +  T'(\L X,\beta \overline{Y}),
     \end{split}\vspace{-0.2cm}
\end{equation*}
\noindent $\overline{P}$ being a $\pi$-vector field defined by
\vspace{-0.1cm}
\begin{equation*}
    g(\overline{P}, \rho Z)=hZ \cdot \sigma(x)\vspace{-0.1cm}
\end{equation*}
and  $T'$ being a 2-form on $TM$, with values in $\pi^{-1}(TM)$,
defined by
  \begin{equation*}\label{eq.t5}
     g(T'(\L X,hY),\rho Z)= g(T(N \rho Z, \rho Y),\rho X).\vspace{-0.1cm}
\end{equation*}
In particular,
\begin{description}

  \item[(a)]  $ \widetilde {\nabla}_{\gamma \overline{X}}\overline{Y}=
    {\nabla}_{ \gamma \overline{X}}\overline{Y}$,

  \item[(b)] $\widetilde {\nabla} _{\widetilde{\beta}\overline{X}}\overline{Y} =
    \nabla _{\beta \overline{X}}\overline{Y}-\nabla _{\L \beta
\overline{X}}\overline{Y}+B(\overline{X}, \overline{Y})
    $.
\end{description}
 The associated curvature tensors  are related by\,{\em:}
 \begin{description}
  \item[$\textbf{(a)}'$] $\widetilde {S}(\overline{X},\overline{Y})\overline{Z}=
S(\overline{X},\overline{Y})\overline{Z}$.

  \item[$\textbf{(b)}'$] $ \widetilde {P}(\overline{X},\overline{Y})\overline{Z}=
  P(\overline{X},\overline{Y})\overline{Z}+
  V(\overline{X},\overline{Y})\overline{Z},$

  \item [$\textbf{(c)}'$]
 $\widetilde {R}(\overline{X}, \overline{Y})\overline{Z} = R(\overline{X},\overline{Y})\overline{Z}
 +H(\overline{X},\overline{Y})\overline{Z}$,
\end{description}
\noindent where $H$ and $V$ are the $\pi$-tensor fields
defined by\vspace{-0.2cm}
\begin{equation*}
    \begin{split}
      V(\overline{X},\overline{Y})\overline{Z} =  (\nabla_{\gamma \overline{Y}}B)(\overline{X},\overline{Z})
     + B(T( \overline{Y},\overline{X}),\overline{Z})-S(N\overline{X}, \overline{Y})
     \overline{Z},
    \end{split}
\end{equation*}
\vspace{-0.1cm}
  \begin{equation*}
    \begin{split}
    H(\overline{X}, \overline{Y})\overline{Z} =&
    S(N\overline{X} ,N \overline{Y})\overline{Z}
    - \mathfrak{U}_{\overline{X},\overline{Y}} \{
    (\nabla_{\beta \overline{X}}B)( \overline{Y},\overline{Z}) -(\nabla_{\L \beta \overline{X}}B)( \overline{Y},\overline{Z})
\\
         & +P(\overline{X}, N \overline{Y})\overline{Z}
                  +B( \overline{X},B(\overline{Y},\overline{Z}))- B(T(N
         \overline{X},\overline{Y}),\overline{Z})\};
          \end{split}
\end{equation*}
$B$ being defined by (\ref{ABN}) and $\L$ by (\ref{eq.1a}).
\end{thm}

\begin{thm}\label{th.4B} If $(M,L)$ and $(M,\widetilde {L})$ are conformally related Finsler
manifolds, then the associated Berwald  connections $D^{\circ}$ and
  $\widetilde{D^{\circ}}$  are related by:\vspace{-0.2cm}
\begin{equation*}\label{eq.2a}
    \widetilde{D^{\circ}}_{X}\overline{Y}=D^{\circ}_{X}\overline{Y}+\omega^{\circ}(X,\overline{Y} ),\vspace{-0.2cm}
\end{equation*}

\noindent  where\ \ $\omega^{\circ}(X,\overline{Y})=K([\gamma
\overline{Y}, \L]X)+D^{\circ}_{\L X}\overline{Y}$
\vspace{4pt}\\
\noindent In particular, we have \vspace{-0.1cm}
\begin{description}

  \item[(a)]$\widetilde{D^{\circ}}_{\gamma \overline{X}}\overline{Y}=D^{\circ}_{\gamma \overline{X}}\overline{Y}$

  \item[(b)] $\widetilde{D^{\circ}}_{\widetilde{\beta} \overline{X} }\overline{Y}=
    D^{\circ}_{\beta \overline{X}}\overline{Y}- D^{\circ}_{\gamma N \o X}\o Y
    +B^{\circ}( \o X, \o Y)$.
\end{description}
 The associated curvature tensors  are related by\,{\em:}
\begin{description}

  \item[$\textbf{(a)}'$] $\widetilde{S^{\circ}}(\overline{X},\overline{Y})\overline{Z}=
S^{\circ}(\overline{X},\overline{Y})\overline{Z}=0 .$

  \item[$\textbf{(b)}'$]  $ \widetilde{P^{\circ}}(\overline{X},\overline{Y})\overline{Z}=
   P^{\circ}(\overline{X},\overline{Y})\overline{Z} + (D^{\circ}_{\gamma \overline{Y}}B^{\circ})(\overline{X},\overline{Z}).$

  \item[$\textbf{(c)}'$]  $\widetilde{R^{\circ}}(\overline{X}, \overline{Y})\overline{Z} =
  R^{\circ}(\overline{X},\overline{Y})\overline{Z}
 + \mathfrak{U}_{\overline{X},\overline{Y}}
 \{(D^{\circ}_{\gamma N \overline{X}}B^{\circ})(\overline{Y},\overline{Z})
     - (D^{\circ}_{\beta
     \overline{X}}B^{\circ})(\overline{Y},\overline{Z})$\\
     ${{\qquad\qquad\quad}}+
         P^{\circ}(\overline{Y}, N \overline{X})\overline{Z} -B^{\circ}( \overline{X},B^{\circ}(\overline{Y},\overline{Z}))\},$
\end{description}
\noindent  where \,\, $ B^{\circ}( \o X, \o Y):=\omega^{\circ}(\beta
\overline{X},\overline{Y}).$
\end{thm}

Now, we  turn our attention to the Chern and Hashigauchi
connections.\vspace{-0.2cm}
\begin{thm}\label{th.4}  Let $(M,L)$ and $(M,\widetilde {L})$ be  conformally related  Finsler manifolds with
$\widetilde{g}=e^{2 \sigma(x)}g $.  The associated
  Chern connections  $D^{c}$ and  $\widetilde{D^{c}}$ are related by
 \begin{equation}\label{eq.2a}
    \widetilde{D^{c}}_{X}\overline{Y}=D^{c}_{X}\overline{Y}+{\omega}^{c}(X,\overline{Y} ),\vspace{-0.2cm}
 \end{equation}
  where \vspace{-0.1cm}
 \begin{equation*}\label{eq.t6}
    \begin{split}
     {\omega}^{c}(X,\overline{Y}):=&(h X \cdot \sigma (x)) \overline{Y}
     + (\beta \overline{Y} \cdot \sigma (x)) \rho X -g(\rho X,
     \overline{Y}) \overline{P} \\
     & - T(N \overline{Y},\rho X) +  T'(\L X,\beta \overline{Y})- T(N \rho X,\overline{Y}) ,
     \end{split}\vspace{-0.2cm}
\end{equation*}
   \vspace{4pt}
 \noindent In particular, we have \vspace{-0.1cm}
\begin{description}

  \item[(a)]$\widetilde{D^{c}}_{\gamma \overline{X}}\overline{Y}=D^{c}_{\gamma \overline{X}}\overline{Y}$

  \item[(b)] $\widetilde{D^{c}}_{\widetilde{\beta} \overline{X} }\overline{Y}=
    D^{c}_{\beta \overline{X}}\overline{Y}- D^{c}_{\gamma N
\overline{X}}\o Y +{B^{c}}(\overline{X},\overline{Y})$,
\end{description}
where $B^{c}(\overline{X},
\overline{Y}):=\omega^{c}(\beta\overline{X}, \overline{Y})$.
 \end{thm}

\begin{proof} Formula (\ref{eq.2a}) follows from Theorem
\ref{th.r5}, Theorem \ref{th.c1} and Lemma  \ref{lem.c1}, taking
into account the fact that the (h)hv-torsion tensor $T$ is
conformally invariant \cite{r62}.
\par
In more details, \vspace{-0.2cm}
\begin{equation*}
    \left.
    \begin{array}{rcl}
       \widetilde{D^{c}}_{X}\overline{Y} &=&\widetilde{\nabla} _{X}\overline{Y}- \widetilde{T}(\widetilde{K}X,\overline{Y}) \\
       & =& \nabla _{X}\overline{Y} + \omega(X,\overline{Y})-T(KX,\overline{Y}) - T(K \L X,\overline{Y})\\
       & = &D^{c}_{X}\overline{Y}+
{\omega}^{c}(X,\overline{Y}).\vspace{-0.2cm}
         \end{array}
         \right.
\end{equation*}
 Relations  \textbf{(a)} and  \textbf{(b)} follow from (\ref{eq.2a})
by setting $X=\gamma \overline{X}$ and  $X=\widetilde{\beta}
\overline{X}$  respectively.
 \end{proof}

 In view of the above theorem, we have \vspace{-0.2cm}
\begin{thm}\label{tth.6}  Under a Finsler conformal change
$\widetilde{g}=e^{2 \sigma(x)}g $, we have\vspace{-0.2cm}
\begin{description}

  \item[(a)] $\widetilde{S^{c}}(\overline{X},\overline{Y})\overline{Z}=
S^{c}(\overline{X},\overline{Y})\overline{Z}=0,$

  \item[(b)]  $ \widetilde{P^{c}}(\overline{X},\overline{Y})\overline{Z}=
   P^{c}(\overline{X},\overline{Y})\overline{Z} + (D^{c}_{\gamma \overline{Y}}B^{c})(\overline{X},\overline{Z}),$

  \item[(c)]  $\widetilde{R^{c}}(\overline{X}, \overline{Y})\overline{Z} =
  R^{c}(\overline{X},\overline{Y})\overline{Z}
 + \mathfrak{U}_{\overline{X},\overline{Y}} \{(D^{c}_{\gamma N \overline{X}}B^{c})(\overline{Y},\overline{Z})
     - (D^{c}_{\beta
     \overline{X}}B^{c})(\overline{Y},\overline{Z})$\\
     ${\qquad\qquad\quad} +
      P^{c}(\overline{Y}, N \overline{X})\overline{Z}
      -B^{c}( \overline{X},B^{c}(\overline{Y},\overline{Z}))\}.$
 \end{description}
 \end{thm}

\begin{thm}\label{th.h4}  Let $(M,L)$ and $(M,\widetilde {L})$ be  conformally related  Finsler manifolds with
$\widetilde{g}=e^{2 \sigma(x)}g $.  The associated
  Hashiguchi connections  $D^{*}$ and  $\widetilde{D^{*}}$ are related by
 \begin{equation}\label{eq.h2a}
   \widetilde{ D^{*}}_{X}\overline{Y}=D^{*}_{X}\overline{Y}+{\omega}^{*}(X,\overline{Y} ),\vspace{-0.2cm}
 \end{equation}
  \noindent  where\ \ ${\omega}^{*}(X,\overline{Y})=(D^{*}_{\gamma \o Y}N)(\rho X)
  +NT(\overline{Y},\rho X).$\vspace{4pt}\\
 \noindent In particular, we have \vspace{-0.1cm}
\begin{description}

  \item[(a)]$\widetilde{D^{*}}_{\gamma \overline{X}}\overline{Y}=D^{*}_{\gamma \overline{X}}\overline{Y}$

  \item[(b)] $\widetilde{D^{*}}_{\widetilde{\beta} \overline{X} }\overline{Y}=
    D^{*}_{\beta \overline{X}}\overline{Y}- D^{*}_{\gamma N \overline{X}}\o Y
    +B^{*}(\overline{X}, \o Y) $,
\end{description}
where $B^{*}(\overline{X}, \overline{Y}):=\omega^{*}(\beta
\overline{X}, \overline{Y})$.
 \end{thm}

\begin{proof}
Formula (\ref{eq.h2a}) follows from Theorem \ref{th.t12} and Theorem
\ref{th.c1}$\textbf{(b)}'$. \\
In more details,
\begin{equation*}
    \begin{split}
       \widetilde{D^{*}}_{X}\overline{Y} &=\widetilde{\nabla} _{X}\overline{Y}
       +\widetilde{P}(\rho X ,\overline{Y})\overline{\eta} \\
       & = \nabla _{X}\overline{Y} + B(\rho X,\overline{Y}) +
         P(\rho X ,\overline{Y})\overline{\eta}+ V(\rho X ,\overline{Y})\overline{\eta}\\
         &= D^{*}_{X}\overline{Y}+ B(\rho X,\overline{Y})+\nabla_{\gamma \o Y}B(\rho X,\o \eta)
         -B(\nabla_{\gamma \o Y}\rho X,\o \eta)\\
         &- B(\rho X,\overline{Y})+B(T(\o Y,\rho X),\o \eta)\\
                   &= D^{*}_{X}\overline{Y}+
(\nabla_{\gamma \o Y}N)(\rho X)+NT(\o Y,\rho X).\\
 &= D^{*}_{X}\overline{Y}+
(D^{*}_{\gamma \o Y}N)(\rho X)+NT(\o Y,\rho X).\\
        &= D^{*}_{X}\overline{Y}+
{\omega}^{*}(X,\overline{Y}).
    \end{split}
\end{equation*}
Relation \textbf{(a)} follows from (\ref{eq.h2a}) by setting
$X=\gamma \overline{X}$ noting that $\rho\circ\gamma=0$, whereas
relation \textbf{(b)}
 follows  from the same formula by setting  $X=\widetilde{\beta}
\overline{X}$, noting that $\widetilde{\beta}=\beta-\L \circ \beta$.
 \end{proof}

\begin{thm}\label{Th.6}   Under a Finsler conformal change
$\widetilde{g}=e^{2 \sigma(x)}g $, we have\vspace{-0.2cm}
 \begin{description}
  \item[(a)] $\widetilde{S^{*}}(\overline{X},\overline{Y})\overline{Z}=
{S^{*}}(\overline{X},\overline{Y})\overline{Z}$,

  \item[(b)] $ \widetilde{P^{*}}(\overline{X},\overline{Y})\overline{Z}=
  {P^{*}}(\overline{X},\overline{Y})\overline{Z}-{S^{*}}(N\overline{X}, \overline{Y})
      \overline{Z} + ({D^{*}}_{\gamma \overline{Y}}{B^{*}})(\overline{X},\overline{Z})
     + {B^{*}}(T( \overline{Y},\overline{X}),\overline{Z}),$

  \item [(c)]
 $\widetilde{R^{*}}(\overline{X}, \overline{Y})\overline{Z}
  = {R^{*}}(\overline{X},\overline{Y})\overline{Z}
 +{S^{*}}(N\overline{X} ,N \overline{Y})\overline{Z}
    - \mathfrak{U}_{\overline{X},\overline{Y}} \{{P^{*}}(\overline{X}, N \overline{Y})\overline{Z}
         +(D^{*}_{\beta \overline{X}}B^{*})( \overline{Y},\overline{Z})$ \\
         ${\qquad\qquad\quad} -(D^{*}_{\gamma N\overline{X}}B^{*})( \overline{Y},\overline{Z})
         +B^{*}( \overline{X},B^{*}(\overline{Y},\overline{Z}))-
         B^{*}(T(N \overline{X},\overline{Y}),\overline{Z})\}$.
\end{description}
\end{thm}


\Section{Conformal change of special Finsler spaces}

In this section, we  establish \emph{an intrinsic investigation}
of the  conformal change of the most imortant special Finsler
spaces. Moreover, we obtain necessary and sufficient  conditions
for such special Finsler spaces to be conformally invariant.
\par
Throughout this section, $g$, $\widehat{g}$, $\nabla$ and
$D^{\circ}$ denote respectively the Finsler metric on
$\pi^{-1}(TM)$, the induced metric on $\pi^{-1}(T^{*}M)$, the
Cartan connection and the Berwald connection associated with a
Finsler manifold $(M,L)$. Also, $R$, $P$ and $S$ denote
respectively the h-, hv- and v-curvature tensors  of Cartan
connection, whereas $R^{\circ}$, $P^{\circ}$ and $S^{\circ}$
denote respectively the h-, hv- and v-curvature tensors of Berwald
connection. Finally, $T$ denotes the (h)hv-torsion tensor of
Cartan connection.

\bigskip

We first set the intrinsic definitions of the special Finsler spaces
that will be treated. These definitions are quoted from \cite{r86},
where we have made a systematic intrinsic study of special Finsler
spaces.

\begin{defn}\label{def.1a} A Finsler manifold $(M,L)$ is{\em:}
\begin{description}
  \item[(a)]  Riemannian  if the metric tensor $g(x,y)$ is independent
  of $y$ or, equivalently, if
  $$T=0.$$

  \item[(b)] locally Minkowskian  if the metric tensor $g(x,y)$ is independent
  of $x$ or,  equivalently, if\vspace{6pt}
  $$\nabla_{\beta \overline{X}}\,T  =0\,\,\, \text{and}\,\,\,  R=0.$$

  The above conditions are also equivalent to\vspace{6pt}
  $$\widehat{R}=0\,\,\, \text{and}\,\,\,P^{\circ}=0.$$
\end{description}
\end{defn}

\begin{defn}\label{def.1aa} A Finsler manifold $(M,L)$ is said to be{\em:}
\begin{description}

 \item[(a)]  Berwald  if the torsion tensor $T$ is horizontally
  parallel. That is,
  $$\nabla_{\beta \overline{X}}\,T =0.$$

\item[(b)]  $C^h$-recurrent  if the torsion tensor
 $T$  satisfies the condition
  $$\nabla_{\beta \overline{X}}\,T
 =\lambda_{o} (\overline{X})\,T,$$
 \noindent where $\lambda_{o}$ is a $\pi$-form  of order one.
\item[(c)]
  ${P}^{*}$-Finsler manifold
  if the
$\pi$-tensor field  $\nabla_{\beta \overline{\eta}}T$
  is expressed in the form
  $$\nabla_{\beta \overline{\eta}}\,T= \lambda (x,y)\,T,$$
  where $\lambda (x,y)=
  \frac{\widehat{g}(\nabla_{\beta\overline{\eta}}\,C,C)}{C^2}$  and
  $C^2:=\widehat{g}(C,C)=C(\overline{C})\neq 0$.

\end{description}
\end{defn}

\begin{defn}A Finsler manifold $(M,L)$ is said to be{\em:}
\begin{description}

    \item[(a)]  $C^v$-recurrent  if the torsion tensor
 $T$  satisfies the condition
  $$\nabla_{\gamma \overline{X}}T
 =\lambda_{o} (\overline{X}) T.$$

    \item[(b)]  $C^0$-recurrent  if the torsion tensor
 $T$  satisfies the condition
 $$D^{\circ}_{\gamma \overline{X}}T
 =\lambda_{o} (\overline{X}) T.$$

\end{description}
\end{defn}

\begin{defn}\label{3.def.1} A Finsler manifold $(M,L)$ is said to
be{\em:}
\begin{description}

 \item[(a)]  semi-$C$-reducible if $dim\, M \geq 3$ and the
  Cartan tensor $T$ has the form
  \begin{equation*}
  \begin{split}
    T(\overline{X},\overline{Y},\overline{Z})= & \frac{\mu}{n+1} \{\hbar(\overline{X}
  ,\overline{Y})C(\overline{Z})+\hbar(\overline{Y}
  ,\overline{Z})C(\overline{X})+\hbar(\overline{Z},\overline{X}) C(\overline{Y})\} \\
      &+ \frac{\tau}{C^{2}}C(\overline{X}) C(\overline{Y})
  C(\overline{Z}),
  \end{split}
  \end{equation*}
  \end{description}
where $\mu$ and $\tau$ are scalar functions on $TM$ satisfying $\mu
+\tau=1$.
\begin{description}
\item[(b)] $C$-reducible   if $dim\, M \geq 3$ and the
  Cartan tensor $T$ has the form
  \begin{equation*}
      T(\overline{X},\overline{Y},\overline{Z})=  \frac{1}{n+1} \{\hbar(\overline{X}
  ,\overline{Y})C(\overline{Z})+\hbar(\overline{Y}
  ,\overline{Z})C(\overline{X})+\hbar(\overline{Z},\overline{X})
  C(\overline{Y})\}.
   \end{equation*}

\item[(c)] $C_{2}$-like   if $dim\, M \geq 2$ and  the
  Cartan tensor $T$ has the form
  \begin{equation*}
    T(\overline{X},\overline{Y},\overline{Z})= \frac{1}{C^{2}}C(\overline{X}) C(\overline{Y})
  C(\overline{Z}).
  \end{equation*}

  \end{description}
\end{defn}

\begin{defn} \label{3.def.2} A Finsler manifold $(M,L)$, where  $dim\, M \geq 3$, is said to be
  quasi-$C$-reducible
if  the Cartan tensor $T$ is written as{\em:}
  $$T(\overline{X},\overline{Y},\overline{Z})= A(\overline{X}
  ,\overline{Y})C(\overline{Z})+A(\overline{Y}
  ,\overline{Z})C(\overline{X})+A(\overline{Z}
  ,\overline{X})C(\overline{Y}),$$
where $A$ is a symmetric indicatory  $(2)\,\pi$-form
{\em$(A(\overline{X},\overline{\eta})=0$ for all $\overline{X}$)}.
\end{defn}

\begin{defn}\label{3.def.3} A Finsler manifold $(M,L)$ is said to be{\em:}
\begin{description}
\item[(a)]  $S_{3}$-like if $dim\, M\geq 4$ and the vertical curvature tensor
$S$ has the form{\em:}
$$S(\overline{X},\overline{Y},\overline{Z},\overline{W})=
\frac{Sc^{v}}{(n-1)(n-2)} \{
\hbar(\overline{X},\overline{Z})\hbar(\overline{Y},\overline{W})-\hbar(\overline{X},\overline{W})
\hbar(\overline{Y},\overline{Z}) \}.$$

\item[(b)]  $S_{4}$-like  if $dim\, M\geq 5$ and the vertical curvature tensor
$S$ has the form{\em:}
\begin{equation}\label{h}
   \begin{split}
    S(\overline{X},\overline{Y},\overline{Z},\overline{W})=
    &
   \hbar(\overline{X},\overline{Z})\textbf{F}(\overline{Y},\overline{W})
   -\hbar(\overline{Y},\overline{Z})\textbf{F}(\overline{X},\overline{W}) \\
       &+\hbar(\overline{Y},\overline{W})\textbf{F}(\overline{X},\overline{Z})
        -
        \hbar(\overline{X},\overline{W})\textbf{F}(\overline{Y},\overline{Z}),
   \end{split}
\end{equation}
\end{description}
\noindent where $\textbf{F}= {\displaystyle\frac{1}{n-3}\{Ric^v-
\frac{Sc^v\, \hbar}{2(n-2)}\}}$.
\end{defn}

\begin{defn}\label{def.a1} A Finsler manifold $(M,L)$ is said to be
 $S^v$-recurrent  if the $v$-curvature  tensor
 $S$  satisfies the condition
  $$(\nabla_{\gamma \overline{X}}S)(\overline{Y},\overline{Z},\overline{W})
  =\lambda(\overline{X})S(\overline{Y},\overline{Z})\overline{W},$$
 where $\lambda$ is a $\pi$-form of order one.

\end{defn}

\begin{defn}\label{def.2a} A Finsler manifold $(M,L)$ is said to
be{\em:}
\begin{description}
\item[(a)] a Landsberg manifold  if
$$\widehat{P}=0,\,  \text {or equivalently}\,\,
 \nabla_{\beta \overline{\eta}}\,T =0.$$

 \item[(b)] a general Landsberg manifold if $$Tr\{
  \overline{Y} \To \widehat{P}(\overline{X},\overline{Y})\}=0\,\,\,
 \forall\,  \overline{X} \in\cp,
  \text {\,or equivalently \,} \nabla _{\beta \overline{\eta}}\ C =0.$$
  \end{description}
\end{defn}


\begin{defn} \label{p-sy} A Finsler manifold $(M,L)$ is said to be
$P$-symmetric if the mixed curvature  tensor $P$ satisfies
$$P(\overline{X},\overline{Y})\overline{Z}
 =P(\overline{Y},\overline{X})\overline{Z}, \ \ \forall\  \overline{X}, \overline{Y}, \overline{Z}\in\cp  .$$
\end{defn}

\begin{defn}\label{def.p2like}  A Finsler manifold $(M,L)$, where  $dim\, M\geq 3$,
is said to be $P_{2}$-like if the mixed curvature tensor $P$ has the
form{\em:}
$$P(\overline{X},\overline{Y},\overline{Z}, \overline{W})=
 \alpha(\overline{Z})T (\overline{X},\overline{Y},\overline{W})
 -\alpha(\overline{W})\, T(\overline{X},\overline{Y},\overline{Z}),\vspace{-0.2cm}$$ where $\alpha$ is
  a $(1)\,\pi$-form, positively homogeneous of degree $0$.
  \end{defn}

\begin{defn}\label{def.p1} A Finsler manifold $(M,L)$, where  $dim\, M\geq 3$,
is said to be $P$-reducible
  if  the $\pi$-tensor field
 $P(\overline{X},\overline{Y},\overline{Z}):
  =g(\widehat{P}(\overline{X},\overline{Y}),\overline{Z})$ is expressed
  in the form{\em:}
 $$P(\overline{X},\overline{Y},\overline{Z})=\delta(\overline{X})\hbar (\overline{Y},\overline{Z})
  +\delta(\overline{Y})\hbar (\overline{Z},\overline{X})+ \delta(\overline{Z})\hbar
  (\overline{X},\overline{Y}),$$ where $\delta$ is
  the $\pi$-form defined by  $\delta=\frac{1}{n+1}\nabla_{\beta \overline{\eta}}\,C
  .$
\end{defn}


\begin{defn}\label{def.2} A Finsler manifold $(M,L)$, where $\dim M\geq 3$,
is said to be  $h$-isotropic if there exists a scalar $k_{o}$ such
that the horizontal curvature tensor $R$ has the form
$$R(\overline{X},\overline{Y})\overline{Z}=k_{o} \{g(\overline{X},\overline{Z})
\overline{Y}-g(\overline{Y},\overline{Z})\overline{X}\}.$$
\end{defn}

\begin{defn}\label{def.2} A Finsler manifold $(M,L)$, where  $dim\, M\geq 3$, is said to be{\em:}
\begin{description}

 \item[(a)]  of scalar curvature  if there exists a  scalar function $k: \T M \To
 \Real$ such that the horizontal curvature tensor $R$ satisfies
the relation
 $$R(\overline{\eta},\overline{X},\overline{\eta},\overline{Y})=
  k L^{2} \hbar(\overline{X},\overline{Y}).$$

\item[(b)]  of constant curvature  if the function $k$ in
{\emph{\textbf{(a)}}} is constant.
\end{description}
\end{defn}


\begin{defn}\label{def.r3} A Finsler manifold $(M,L)$ is said to be $R_{3}$-like if
 $dim\, M\geq 4$ and the horizontal
curvature tensor $R$ is expressed in the form
\begin{equation}\label{h}
   \begin{split}
    R(\overline{X},\overline{Y},\overline{Z},\overline{W})= &
       g(\overline{X},\overline{Z})F(\overline{Y},\overline{W})
   -g(\overline{Y},\overline{Z})F(\overline{X},\overline{W})  \\
       &+ g(\overline{Y},\overline{W})F (\overline{X},\overline{Z})
       - g(\overline{X},\overline{W})F(\overline{Y},\overline{Z}) ,
   \end{split}
\end{equation}
\end{defn}\vspace{-0.1cm}
\noindent where $F$ is the $(2)\pi$-form defined by $ F={
\frac{1}{n-2}\{ Ric^h- \frac{Sc^h\, g}{2(n-1)}\}}$.

\begin{defn} A Finsler manifold $(M,L)$ is called of perpendicular
scalar {\em({\it simply, $p$-scaler})} curvature if the
$h$-curvature tensor $R$ satisfies the condition
\begin{eqnarray*}
        R(\phi(\overline{X}),\phi(\overline{Y}),\phi(\o Z),\phi(\o W))
          &=&R_{o} \{ \hbar(\overline{X}, \o Z)\hbar(\o Y,\o W)
    - \hbar(\overline{X},\o W)\hbar(\o Y, \o Z) \},
\end{eqnarray*}
where $R_{o}$ is a function on $TM$,  called  perpendicular scaler
curvature, and $\phi$ is the $\pi$-tensor field defined  by $
\phi(\overline{X}):=\overline{X}-L^{-1}\ell(\overline{X})\o \eta$.
\end{defn}

\begin{defn}\label{def.p2} A Finsler manifold $(M,L)$ is called of $s$-$ps$ curvature if
$(M,L)$ is both of scalar curvature and of $p$-scalar curvature.
\end{defn}

\begin{defn} \label{3.def.4}A Finsler manifold $(M,L)$ is said to be  symmetric if
the h-curvature tensor ${R^{\circ}}$  of the Berwald connection
$D^{\circ}$ is horizontally parallel{\em:} $D^{\circ}_{\beta
\overline{X}}{R^{^{\circ}}}=0 .$
\end{defn}

Now, we focus our attention to the change of the above mentioned
 special Finsler manifolds under a conformal transformation
  $g\To \widetilde{g}=e^{2 \sigma(x)}g $.
\textbf{ In what follows we assume that the Finsler manifolds $(M,L)$
 and $(M,\widetilde{L})$ are conformally related.}

\begin{prop}\label{p.3}~\par\vspace{-0.1cm}
\begin{description}
  \item[(a)] $(M,L)$ is  a Riemaniann manifold if, and only if,
   $(M,\widetilde{L})$ is a Riemaniann manifold.

  \item[(b)]Assume that  $ D^{\circ}_{\gamma \overline{X}}{B^{\circ}}=0$
   and $H(\overline{X},\o Y)\o \eta=0$. Then,
   $(M,L)$ is    Locally Minkowskian  if, and only
if, $(M,\widetilde{L})$ is   Locally Minkowskian.
\end{description}
\end{prop}
\begin{proof}~\par
\vspace{6pt} \noindent\textbf{(a)} Follows from  Definition
\ref{def.1a} together with
    the fact that the $(h)hv$-torsion tensor $T$ is conformally invariant.

\vspace{6pt} \noindent\textbf{(b)}  By Theorem
\ref{th.c1}$\textbf{(c)}'$ and Theorem \ref{th.4B}$\textbf{(b)}'$,
we get\vspace{-0.1cm}
$$\widetilde{R}(\overline{X},\o Y)\o \eta=R(\overline{X},\o Y)\o \eta, \quad \text{and}   \quad
  \widetilde{P^{\circ}} (\overline{X},\o Y)\o Z=P^{\circ}(\overline{X},\o Y)\o Z. \vspace{-0.1cm}       $$
The result  follows  then from  Definition \ref{def.1a}.
\end{proof}

 Let us introduce the $\pi$-tensor field \vspace{-4pt}
\begin{equation}\label{Eq.A}
   \A (\overline{X},\overline{Y},\overline{Z}):=
   T(U(\beta \overline{X},\overline{Y}), \overline{Z})+ T(U(\beta
  \overline{X},\overline{Z}), \overline{Y})-U(\beta \overline{X},T(\overline{Y},\overline{Z})), \vspace{-4pt}
\end{equation}
  where $U(\beta \overline{X}, \overline{Y}):=B(\overline{X}, \overline{Y})-
  \nabla _{\L \beta
\overline{X}}\overline{Y}$.
  \par
  One can show that the  $\pi$-tensor field
$\A$ has the property that $\A
(\overline{X},\overline{Y},\overline{\eta})= 0$.

\begin{prop}\label{p.4}Assume that the $\pi$-tensor field $\A $ vanishes. Then, $(M,L)$ is a
Berwald {\em({\it resp.  $C^{h}$-recurrent})} manifold if, and only
if, $(M,\widetilde{L})$ is a Berwald {\em({\it resp.
$C^{h}$-recurrent})} manifold.
\end{prop}

 \begin{proof}
  Using Theorem \ref{th.c1}\textbf{(b)}, taking into account the
 fact that $T$ is conformally invariant,
 we get\vspace{-0.1cm}
\begin{equation*}\label{}
    \begin{split}
      (\widetilde{\nabla}_{\widetilde {\beta}\overline{X}}\widetilde{T})(
      \overline{Y},\overline{Z})& =  \nabla_{\beta \overline{X}}T(\overline{Y},\overline{Z})-
        T(\nabla_{\beta \overline{X}}\overline{Y},\overline{Z})-
        T(\overline{Y},\nabla_{\beta \overline{X}}\overline{Z})\\&
        -\{T(U(\beta \overline{X},\overline{Y}), \overline{Z})+ T(\overline{Y},U(\beta
  \overline{X},\overline{Z}))-U(\beta \overline{X},T(\overline{Y},\overline{Z}))\}.
    \end{split}\vspace{-0.1cm}
\end{equation*}
Consequently,\vspace{-0.1cm}
\begin{equation}\label{eq.8}
(\widetilde{\nabla}_{\widetilde {\beta}\overline{X}}\widetilde{T})(
      \overline{Y},\overline{Z})=(\nabla_{\beta \overline{X}}T)(
      \overline{Y},\overline{Z})- \A
      (\overline{X},\overline{Y},\overline{Z}).\vspace{-0.1cm}
\end{equation}
Hence, under the given assumption, we have\vspace{-0.1cm}
\begin{equation}\label{.eq.8}\widetilde{\nabla}_{\widetilde
{\beta}\overline{X}}\widetilde{T}=\nabla_{\beta
\overline{X}}T.\vspace{-0.1cm}
\end{equation}
 Therefore, $(M,L)$ is Berwald iff $(M,\widetilde{L})$ is  Berwald.
\par
 On the other hand, if $(M,L)$ is $C^{h}$-recurrent, then the
 $(h)hv$-torsion
 tensor $ T$  has the property that
 $\nabla_{\beta \overline{X}}T =\lambda_{o} (\overline{X})
 T$, where $\lambda_{o}$ is a $\pi$-form. \\
  Now, from  (\ref{.eq.8}), we obtain\vspace{-0.1cm}
 $$\widetilde{\nabla}_{\widetilde{\beta }\overline{X}}\widetilde{T}
={\lambda}_{o}(\overline{X}) \widetilde{T}.\vspace{-0.1cm}$$ This
implies that $(M,\widetilde{L})$ is $C^{h}$-recurrent. The converse
can be proved similarly.
\end{proof}

\begin{prop}\label{p.5}
 Assume that the $\pi$-tensor field $\A $ has the property that $i_{\overline{\eta}}\,\A =0 $.
Then, $(M,L)$ is a ${P}^*$-Finsler manifold  if, and only if,
$(M,\widetilde{L})$ is a ${P}^*$-Finsler manifold.
\end{prop}

\begin{proof}
 From  relation (\ref{eq.8}), we have\vspace{-0.1cm}
$$\nabla_{\beta \o \eta}T=\widetilde{\nabla}_{\widetilde{\beta}\o \eta}\widetilde{T}.\vspace{-0.1cm}$$
 Hence, the
$\pi$-tensor field  $\nabla_{\beta \o \eta}C$ is conformally
invariant. This, together with the fact that $\widetilde{C}=C$,
imply that the scalar function $\lambda(x,y)$ defined by
 $\lambda(x,y):=\frac{\widehat{g}(\nabla_{\beta \o \eta}C, C)}{\widehat{g}(C,C)} $
 is  also conformally invariant. Hence the result.
\end{proof}

\begin{prop}\label{.p.4} A Finsler manifold  $(M,L)$ is  $C^{v}$-recurrent
{\em({\it resp.  $C^{0}$-recurrent})} if, and only if,
$(M,\widetilde{L})$ is   $C^{v}$-recurrent {\em({\it resp.
$C^{0}$-recurrent})}.
\end{prop}

\begin{proof}
If $(M,L)$ is $C^{v}$-recurrent, then the $(h)hv$-torsion tensor $
T$  has the form  $\nabla_{\gamma \overline{X}}T =\lambda_{o}
(\overline{X})
 T$, where $\lambda_{o}$ is a $\pi$-form.
 Since the map
 $\nabla_{\gamma \overline{X}}: \o Y\longmapsto\nabla_{\gamma \overline{X}} \o Y$
  and the torsion tensor $T$ are conformally invariant, it follows that \vspace{-0.1cm}
 $$\widetilde{\nabla}_{\gamma\overline{X}}\widetilde{T}
={\lambda}_{o}(\overline{X}) \widetilde{T}.\vspace{-0.1cm}$$
 This implies that $(M,\widetilde{L})$ is $C^{v}$-recurrent.
The converse is proved similarly.
\par
The same argument can be applied to the $C^{0}$-recurrence
property.
\end{proof}

\begin{prop}\label{p.6} A Finsler manifold $(M,L)$ is semi-$C$-reducible
  if, and only if, $(M,\widetilde{L})$ is   semi-$C$-reducible.
  Consequently, $(M,L)$ is $C$-reducible {\em({\it resp. $C_{2}$-like})}
  if, and only if, $(M,\widetilde{L})$ is
    $C$-reducible {\em({\it resp. $C_{2}$-like})}.
    \end{prop}

\begin{proof}
 The  semi-$C$-reducibility property is expressed as \vspace{-0.1cm}
 \begin{equation*}
    T(\overline{X},\overline{Y},\overline{Z})=
    \frac{{\mu}}{n+1} \mathfrak{S}_{\overline{X},\o Y,\o Z}\{{\hbar}(\overline{X}
  ,\overline{Y}){C}(\overline{Z})\}+ \frac{\tau}{C^{2}}C(\overline{X}) C(\overline{Y})
  C(\overline{Z}),
 \vspace{-0.1cm}
  \end{equation*}
where $\mu$ and $\tau$ are scalar functions  satisfying
$\mu+\tau=1$ and the symbol $\mathfrak{S}_{\overline{X},\o Y,\o Z}$ denotes cyclic sum over $\overline{X},\o Y$ and $\o Z$  . \\
Since
$\widetilde{C}^{2}:=\widetilde{\widehat{g}}(\widetilde{C},\widetilde{C})=e^{-2
 \sigma}\widehat{g}(C,C)=e^{-2
 \sigma}{C}^{2}$,  $\widetilde{T}(\overline{X},\overline{Y},\overline{Z})= e^{2 \sigma
        (x)}T(\overline{X},\overline{Y},\overline{Z})$ and
the angular metric tensor $\hbar$ is conformally $\sigma$-invariant,
the above relation is equivalent to \vspace{-0.1cm}
 \begin{equation*}
     \widetilde{T}(\overline{X},\overline{Y},\overline{Z})=
     \frac{{\mu}}{n+1} \mathfrak{U}_{\overline{X},\o Y,\o Z}\{\widetilde{\hbar}(\overline{X}
  ,\overline{Y})\widetilde{C}(\overline{Z})\}+ \frac{\tau}{\widetilde{C}^{2}}
  \widetilde{C}(\overline{X}) \widetilde{C}(\overline{Y})
  \widetilde{C}(\overline{Z}) .\vspace{-0.1cm}
   \end{equation*}
Hence, the  semi-$C$-reducibility property is preserved.
 \par
 Finally, the proof of the cases of $C$-reducibility and  $C_{2}$-likeness
is similar.
\end{proof}

\begin{prop}\label{p.7} A Finsler manifold $(M,L)$ is quasi-$C$-reducible
 if, and only if, $(M,\widetilde{L})$ is  quasi-$C$-reducible.
\end{prop}


\begin{thm}\label{th.8}A necessary and sufficient condition for a Finsler  manifold to be conformal to a
Landsberg manifold is that\vspace{-0.1cm}
$$\widehat{P}=i_{\overline{\eta}}\,\A.$$
\end{thm}

\begin{proof}
We have \cite{r96}
 $$ \widehat{P}=\nabla_{\beta
\o \eta}T.$$ From which, together with (\ref{eq.8}), we obtain
\begin{equation}\label{5.eq.2}
     \widetilde{\widehat{P}}-\widehat{P}
     =\widetilde{\nabla}_{\widetilde{\beta} \o \eta}\widetilde{T}
     -\nabla_{\beta \o \eta}T=-i_{\overline{\eta}}\,\A.\vspace{-0.1cm}
\end{equation}
Hence, the result follows.
\end{proof}

\vspace{7pt}

Let us define the $\pi$-tensor field
\begin{equation}\label{Ao}
   \A_{o}(\overline{X}):= Tr\{\o Y \longmapsto
   (i_{\overline{\eta}}\,\A)(\overline{X},\overline{Y})\},
\end{equation}
where $\A$ is the $\pi$-tensor field defined by (\ref{Eq.A}).
\begin{prop}\label{p.8}~\par\vspace{-0.1cm}
\begin{description}
  \item[(a)] Assume that $i_{\overline{\eta}}\,\A=0$.
   Then,  $(M,L)$ is  Landsberg  if, and only if, $(M,\widetilde{L})$
is Landsberg.

 \item[(b)] Assume that $\A_{o}=0$.
  Then,  $(M,L)$ is  general Landsberg  if, and only if, $(M,\widetilde{L})$ is
 general Landsberg.

\end{description}
\end{prop}

\begin{proof}~\par
\vspace{6pt} \noindent\textbf{(a)} Setting $\overline{X}=\o \eta$ in
 (\ref{eq.8}), we get\vspace{-0.1cm}
\begin{equation}\label{eq.x8}
\widetilde{\nabla}_{\widetilde
{\beta}\overline{\eta}}\widetilde{T}=\nabla_{\beta
\overline{\eta}}T- i_{\overline{\eta}}\,\A,\vspace{-0.1cm}
\end{equation}
from which, under the given assumption,  $\nabla_{\beta
\overline{\eta}}T$ is conformally invariant. Hence the result.

\vspace{6pt} \noindent\textbf{(b)} Taking the trace of
(\ref{eq.x8}), we obtain\vspace{-0.1cm}
\begin{equation*}\label{cc}
  \widetilde{\nabla }_{\widetilde{\beta} \overline{\eta}}\widetilde{C} =
\nabla _{\beta \overline{\eta}}C-\A_{o}. \vspace{-0.1cm}
\end{equation*}
From which the result.
\end{proof}

\begin{prop}\label{p.10}
     Assume that $i_{\overline{\eta}}\,\A=0$. Then,  $(M,L)$ is  $P$-reducible
    if, and only if, $(M,\widetilde{L})$ is  ${P}$-reducible.
\end{prop}

\begin{proof} Under a conformal change,
the angular metric tensor $\hbar$ is conformally
$\sigma$-invariant. On the other hand,  $\widehat{P}$ is
conformally invariant by our assumption together with
(\ref{5.eq.2}). Consequently, $\nabla_{\beta \o \eta}T$ is
conformally invariant, which implies that $\nabla_{\beta \o
\eta}C$ (or $\delta$ of Definition \ref{def.p1}) is also
conformally invariant.
\par
Now, since $P(\overline{X},\o Y,\o
Z)={g}({\widehat{P}}(\overline{X},\o Y) ,\o Z)$ is conformally
$\sigma$-invariant, then,  the tensor field
$$\mathbf{U}_{1}(\overline{X},\o Y,\o Z):={g}({\widehat{P}}(\overline{X},\o Y),\o Z)-
\mathfrak{S}_{\overline{X},\o Y,\o Z}\{\delta(\overline{X})\hbar
(\overline{Y},\overline{Z})\} \vspace{-0.1cm}$$ is conformally
$\sigma$-invariant.  From which, the result follows (provided that
$\sigma\neq0$).
\end{proof}

\begin{prop}\label{p.11}
 $(M,L)$ is $S_{3}$-like  {\em({\it resp.  $S_{4}$-like})}
 if, and only if, $(M,\widetilde{L})$ is
   $S_{3}$-like  {\em({\it resp.  $S_{4}$-like})}.
\end{prop}

\begin{proof}
 Let $\mathbb{{U}}$ be the $\pi$-tensor field defined
by \vspace{-0.1cm}
\begin{eqnarray*}
 \mathbb{{U}}(\overline{X},\overline{Y},\overline{Z},\overline{W})
&:=&S(\overline{X},\overline{Y},\overline{Z},\overline{W})-
\frac{Sc^{v}}{(n-1)(n-2)} \{
\hbar(\overline{X},\overline{Z})\hbar(\overline{Y},\overline{W})-\hbar(\overline{X},\overline{W})
\hbar(\overline{Y},\overline{Z})\}.\vspace{-0.1cm}
\end{eqnarray*}
Under a conformal transformation, the $\pi$-tensor field $Sc^{v}
\hbar$  is conformally invariant and $
\widetilde{S}(\overline{X},\overline{Y},\overline{Z},\overline{W})
        = e^{2 \sigma
        (x)}S(\overline{X},\overline{Y},\overline{Z},\overline{W})$
        \cite{r62}.
        Hence, \vspace{-0.1cm}
\begin{eqnarray}\label{eq.x9}
  \widetilde{\mathbb{{U}}}(\overline{X},\overline{Y},\overline{Z},\overline{W})
&=&e^{2 \sigma
(x)}\mathbb{{U}}(\overline{X},\overline{Y},\overline{Z},\overline{W}).\vspace{-0.4cm}
\end{eqnarray}
This means that the $\pi$-tensor field $\mathbb{{U}}$ is conformally
$\sigma$-invariant.
 \par
On the other hand, let $\mathbb{{V}}$ be the $\pi$-tensor field
defined by\vspace{-0.1cm}
\begin{eqnarray*}
 \mathbb{{V}}(\overline{X},\overline{Y},\overline{Z},\overline{W})
&:=&S(\overline{X},\overline{Y},\overline{Z},\overline{W})-
 \hbar(\overline{Z},\overline{X})\textbf{F}(\overline{W},\overline{Y})
   +\hbar(\overline{Z},\overline{Y})\textbf{F}(\overline{W},\overline{X}) \\
      & & -\hbar(\overline{W},\overline{Y})\textbf{F}(\overline{Z},\overline{X})
        +\hbar(\overline{W},\overline{X})\textbf{F}(\overline{Z},\overline{Y}).\vspace{-0.1cm}
\end{eqnarray*}
Since both
 the angular metric tensor $\hbar$ and the $v$-curvature tensor $S$ are conformally $\sigma$-invariant and
$ \textbf{F}(\overline{X},\overline{Y}):=
{\frac{1}{n-3}\{Ric^v(\overline{X},\overline{Y}) - \frac{Sc^v
\hbar(\overline{X},\overline{Y})}{2(n-2)}\}}$  is conformally
invariant, we conclude that
\begin{eqnarray}\label{eq.x10}
 \widetilde{\mathbb{{V}}}(\overline{X},\overline{Y},\overline{Z},\overline{W})
     &=&e^{2 \sigma
(x)}\,\mathbb{{V}}(\overline{X},\overline{Y},\overline{Z},\overline{W}),\vspace{-0.4cm}
\end{eqnarray}
which means that the $\pi$-tensor field $\mathbb{{V}}$ is
conformally $\sigma$-invariant. The result follows from
(\ref{eq.x9}) and (\ref{eq.x10}).
\end{proof}

\begin{prop}\label{p.12}
 $(M,L)$ is  $S^{v}$-recurrent  if, and only if, $(M,\widetilde{L})$ is
  $S^{v}$-recurrent.
\end{prop}

\begin{proof}
Follows from the fact that both the map $\nabla_{\gamma
\overline{X}}: \o Y\longmapsto\nabla_{\gamma \overline{X}} \o Y$
  and  the $v$-curvature
 tensor $ S$ are conformally invariant.
\end{proof}

\begin{prop}\label{p.14}
   Assume that the $\pi$-tensor field $H$ defined in Theorem \ref{th.c1} has the property
   that  $H(\overline{\eta},\overline{X})\overline{\eta}
  = 0$ for all  $\overline{X} \in \mathfrak{X}(\pi(M))$. Then, $(M,L)$ is of  scalar curvature
 if, and only if, $(M,\widetilde{L})$ is  of scalar curvature.
\end{prop}

 \begin{proof}
   By Theorem \ref{th.c1}$\textbf{(c)}'$, we have\vspace{-0.1cm}
   $$  \widetilde
{R}(\overline{\eta},\overline{X},\overline{\eta},\overline{Y})
          = e^{2 \sigma(x)} R(\overline{\eta},\overline{X},
          \overline{\eta},\overline{Y})+
        e^{2
        \sigma(x)}g(H(\overline{\eta},\overline{X})\overline{\eta},\overline{Y}),\vspace{-0.1cm}$$
which implies, by hypothesis, that  \vspace{-0.1cm}
\begin{equation}\label{111}
     \widetilde
{R}(\overline{\eta},\overline{X},\overline{\eta},\overline{Y})
          = e^{2 \sigma(x)} R(\overline{\eta},\overline{X},
          \overline{\eta},\overline{Y}).\vspace{-0.1cm}
\end{equation}
Now, let $(M,L)$ be of  scalar curvature, then the $h$-curvature
tensor $ R$ has the form \vspace{-0.1cm}
$$R(\overline{\eta},\overline{X},\overline{\eta},\overline{Y})=
  k L^{2} \hbar(\overline{X},\overline{Y}). \vspace{-0.1cm}$$
 This, together with (\ref{111}), imply that
  $$ \widetilde{R}(\overline{\eta},\overline{X},\overline{\eta},\overline{Y})
  =e^{2\sigma(x)}kL^{2}\hbar(\overline{X},\overline{Y})=
  e^{-2\sigma(x)}k \widetilde{L}^{2}
  \widetilde{\hbar}(\overline{X},\overline{Y}),$$
 where we have used the fact that both $L^{2}$ and $\hbar$
  are conformally $\sigma$-invariant.
  Hence

$$ \widetilde{R}(\overline{\eta},\overline{X},\overline{\eta},\overline{Y})
=  k_{o} \widetilde{L}^{2}
\widetilde{\hbar}(\overline{X},\overline{Y}),$$
 where $ k_{o}=e^{-2\sigma(x)}k$.
 \end{proof}

\begin{prop}\label{p.15}Assume that the given conformal change is homothetic. Then, we
have\vspace{-0.1cm}
\begin{description}
 \item[(a)] $(M,L)$ is  $P_{2}$-like if, and
only if, $(M,\widetilde{L})$ is ${P}_{2}$-like .

\item[(b)] $(M,L)$ is  $h$-isotropic
 if, and only if, $(M,\widetilde{L})$ is $h$-isotropic .

\item[(c)] $(M,L)$ is of constant curvature
 if, and only if, $(M,\widetilde{L})$   is of constant curvature.

    \item[(d)] $(M,L)$ is of  $p$-scalar curvature
if, and only if, $(M,\widetilde{L})$ is of  $p$-scaler curvature.

 \item[(e)] $(M,L)$ is of  $s$-$ps$-curvature if, and only if,
$(M,\widetilde{L})$ is of  $s$-$ps$-curvature.

    \item[(f)] $(M,L)$ is  $R_{3}$-like if, and
only if, $(M,\widetilde{L})$ is $R_{3}$-like.

 \item[(g)] $(M,L)$ is  symmetric if, and
only if, $(M,\widetilde{L})$ is symmetric.

 \item[(h)] $(M,L)$ is  $P$-symmetric if, and
only if, $(M,\widetilde{L})$ is ${P}$-symmetric.

\end{description}
\end{prop}
\begin{proof}
  The proof  follows from  the fact that:\\
  ${\qquad\qquad\qquad\qquad\qquad}\sigma (x) \, \text{is constant} \Longleftrightarrow \widetilde{\nabla}_{X}\o Y=
 \nabla_{X}\o Y  \,\,\cite{r62}. $
\end{proof}

\bigskip

\par Summing up, the results of this section can be gathered in the
following\vspace{-0.1cm}
\begin{thm}\label{th.9}The following properties are conformally
invariant{\em:}\\
 --\ being Riemannian,\qquad\qquad\qquad\qquad \qquad\qquad  --\ being semi-$C$-reducible, \\
 --\ being $C$-reducible, \qquad\qquad\qquad\qquad  \qquad\qquad --\ being $C_{2}$-like, \\
 --\ being quasi-$C$-reducible, \qquad\qquad\qquad\qquad\quad\, --\ being
 $C^{v}$-recurrent,\\
 --\  being $C^{0}$-recurrent,\qquad\qquad\:\qquad\qquad\qquad\quad\, --\ being $S^{v}$-recurrent,\\
 --\  being $S_{3}$-like,\qquad\qquad\:\qquad\qquad\qquad\quad\qquad\quad\,--\  being $S_{4}$-like. \\
The following properties are conformally invariant under certain
conditions{\em\,:}\\
  --\ being locally Minkowskian,\quad\quad\qquad\qquad \qquad\qquad\quad \, --\ being Berwald,\\
--\ being $C^{h}$-recurrent,\quad\quad\qquad\qquad\qquad\qquad\quad\,\,\,--\ being $P^*$-manifold,\\
--\  being Landsberg,\!\!\quad\quad\qquad\qquad\qquad\qquad\qquad\quad \,--\ being general Landsberg,\\
  --\  being $P$-symmetric,\!\!\!\!\quad\quad\qquad\qquad \qquad\qquad\qquad\,--\ being $P_{2}$-like,\\
 --\ being $P$-reducible,\quad\quad\qquad\qquad \qquad\qquad\qquad\,--\ being $h$-isotropic,\\
  --\ being of scalar curvature,\!\!\quad\quad\qquad\qquad \quad\quad\,\,\,\,\,\,
  --\  being of constant curvature,\\
--\ being
$R_{3}$-like,\,\qquad\quad\quad\qquad\qquad\qquad\qquad\qquad\!\!\,--\
being of $p$-scaler
curvature,\\
 --\   being of $s$-$ps$ curvature,\,\qquad\quad\quad\qquad \qquad\quad\, --\  being symmetric.
\end{thm}

\bigskip

\begin{rem}It should be noted that some important results of
  {\em \cite{r23}, \cite{r5}, \cite{r6}} {\em{(}}obtained in local coordinates{\em{)}} are
retrieved from the above mentioned global results {\em{(}}when
localized{\em{)}}.
\end{rem}

\bigskip
\providecommand{\bysame}{\leavevmode\hbox
to3em{\hrulefill}\thinspace}
\providecommand{\MR}{\relax\ifhmode\unskip\space\fi MR }
\providecommand{\MRhref}[2]{%
  \href{http://www.ams.org/mathscinet-getitem?mr=#1}{#2}
}
\providecommand{\href}[2]{#2}

\end{document}